\documentclass[10pt]{amsart}
\usepackage{amssymb,MnSymbol}
\usepackage{amsthm,amsmath}

\title{On signature-based expressions of system reliability}

\author{Jean-Luc Marichal}
\address{Mathematics Research Unit, FSTC, University of Luxembourg, 6, rue Coudenhove-Kalergi, L-1359 Luxembourg, Luxembourg}
\email{jean-luc.marichal[at]uni.lu }

\author{Pierre Mathonet}
\address{Mathematics Research Unit, FSTC, University of Luxembourg, 6, rue Coudenhove-Kalergi, L-1359 Luxembourg, Luxembourg}
\email{pierre.mathonet[at]uni.lu, p.mathonet[at]ulg.ac.be}

\author{Tam\'as Waldhauser}
\address{Mathematics Research Unit, FSTC, University of Luxembourg, 6, rue Coudenhove-Kalergi, L-1359 Luxembourg, Luxembourg and Bolyai Institute, University of Szeged, Aradi v\'ertan\'uk tere 1, H-6720 Szeged, Hungary}
\email{twaldha[at]math.u-szeged.hu }

\date{February 17, 2011}

\begin{document}

\theoremstyle{plain}
\newtheorem{theorem}{Theorem}
\newtheorem{lemma}[theorem]{Lemma}
\newtheorem{proposition}[theorem]{Proposition}
\newtheorem{corollary}[theorem]{Corollary}
\newtheorem{fact}[theorem]{Fact}
\newtheorem*{main}{Main Theorem}

\theoremstyle{definition}
\newtheorem{definition}[theorem]{Definition}
\newtheorem{example}[theorem]{Example}

\theoremstyle{remark}
\newtheorem*{conjecture}{onjecture}
\newtheorem{remark}{Remark}
\newtheorem{claim}{Claim}

\newcommand{\N}{\mathbb{N}}
\newcommand{\R}{\mathbb{R}}
\newcommand{\Q}{\mathbb{Q}}
\newcommand{\Vspace}{\vspace{2ex}}
\newcommand{\bfx}{\mathbf{x}}

\begin{abstract}
The concept of signature was introduced by Samaniego for systems whose components have i.i.d.\ lifetimes. This concept proved to be
useful in the analysis of theoretical behaviors of systems. In particular, it provides an interesting signature-based representation of
the system reliability in terms of reliabilities of $k$-out-of-$n$ systems. In the non-i.i.d.\ case, we show that, at any time, this representation still
holds true for every coherent system if and only if the component states are exchangeable. We also discuss conditions for
obtaining an alternative representation of the system reliability in which the signature is replaced by its non-i.i.d.\ extension. Finally, we
discuss conditions for the system reliability to have both representations.
\end{abstract}



\maketitle

\section{Introduction}

Consider a system made up of $n$ $(n\geqslant 3)$ components and let $\phi\colon\{0,1\}^n\to\{0,1\}$ be its \emph{structure function}, which
expresses the state of the system in terms of the states of its components. Denote the set of components by $[n]=\{1,\ldots,n\}$. We assume that
the system is \emph{coherent}, which means that $\phi$ is nondecreasing in each variable and has only essential variables, i.e., for every $k\in
[n]$, there exists $\bfx=(x_1,\ldots,x_n)\in\{0,1\}^n$ such that $\phi(\bfx)|_{x_k=0}\neq\phi(\bfx)|_{x_k=1}$.

Let $X_1,\ldots,X_n$ denote the component lifetimes and let $X_{1:n},\ldots,X_{n:n}$ be the order statistics obtained by rearranging the
variables $X_1,\ldots,X_n$ in ascending order of magnitude; that is, $X_{1:n}\leqslant\cdots\leqslant X_{n:n}$. Denote also the system lifetime
by $T$ and the system reliability at time $t>0$ by $\overline{F}_S(t)=\Pr(T>t)$.

Assuming that the component lifetimes are independent and identically distributed (i.i.d.) according to an absolutely continuous joint c.d.f.\
$F$, one can show (see Samaniego \cite{Sam85}) that
\begin{equation}\label{eq:sdf76}
\overline{F}_S(t) = \sum_{k=1}^n\Pr(T=X_{k:n})\,\overline{F}_{k:n}(t)
\end{equation}
for every $t>0$, where $\overline{F}_{k:n}(t)=\Pr(X_{k:n}>t)$.

Under this i.i.d.\ assumption, Samaniego~\cite{Sam85} introduced the \emph{signature} of the system as the $n$-tuple
$\mathbf{s}=(s_1,\ldots,s_n)$, where
$$
s_k=\Pr(T=X_{k:n}), \qquad k\in [n],
$$
is the probability that the $k$th component failure causes the system to fail. It turned out that the signature is a feature of the
system design in the sense that it depends only on the structure function $\phi$ (and not on the c.d.f.\ $F$). Boland~\cite{Bol01} obtained the
explicit formula
$$
s_k=\phi_{n-k+1}-\phi_{n-k}
$$
where
\begin{equation}\label{eq:phikaa}
\phi_k=\frac{1}{{n\choose k}}\,\sum_{\textstyle{\bfx\in\{0,1\}^n\atop |\bfx|=k}}\phi(\bfx)
\end{equation}
and $|\bfx|=\sum_{i=1}^nx_i$. Thus, under the i.i.d.\ assumption, the system reliability can be calculated by the formula
\begin{equation}\label{eq:as897ds}
\overline{F}_S(t) = \sum_{k=1}^n\big(\phi_{n-k+1}-\phi_{n-k}\big)\,\overline{F}_{k:n}(t).
\end{equation}

Since formula (\ref{eq:as897ds}) provides a simple and useful way to compute the system reliability through the concept of signature, it is natural to relax the i.i.d.\ assumption (as Samaniego \cite[Section~8.3]{Sam07} rightly suggested) and search for necessary and sufficient conditions on the joint c.d.f.\ $F$ for formulas
(\ref{eq:sdf76}) and/or (\ref{eq:as897ds}) to still hold for every system design.

On this issue, Kochar et al.~\cite[p.~513]{KocMukSam99} mentioned that (\ref{eq:sdf76}) and (\ref{eq:as897ds}) still hold when the continuous
variables $X_1,\ldots,X_n$ are exchangeable (i.e., when $F$ is invariant under any permutation of indexes); see also \cite{NavRuiSan05,Zha10}
(and \cite[Lemma~1]{NavRyc07} for a detailed proof). It is also noteworthy that Navarro et al.~\cite[Thm.~3.6]{NavSamBalBha08} showed that
(\ref{eq:sdf76}) still holds when the joint c.d.f.\ $F$ has no ties (i.e., $\Pr(X_i=X_j)=0$ for every $i\neq j$) and the variables
$X_1,\ldots,X_n$ are ``weakly exchangeable'' (see Remark~\ref{remarkweak} below). As we will show, all these conditions are not necessary.

Let $\Phi_n$ denote the family of nondecreasing functions $\phi\colon\{0,1\}^n\to\{0,1\}$ whose variables are all essential. In this paper,
without any assumption on the joint c.d.f.\ $F$, we show that, for every $t>0$, the representation in (\ref{eq:as897ds}) of the system
reliability holds for every $\phi\in\Phi_n$ if and only if the variables $\chi_1(t),\ldots,\chi_n(t)$ are exchangeable, where
$$
\chi_k(t)=\mathrm{Ind}(X_k>t)
$$
denotes the \emph{random state} of the $k$th component at time $t$ (i.e., $\chi_k(t)$ is the indicator variable of the event ($X_k>t$)). This
result is stated in Theorem~\ref{thm:aasd78}.

Assuming that the joint c.d.f.\ $F$ has no ties, we also yield necessary and sufficient conditions on $F$ for formula (\ref{eq:sdf76}) to hold
for every $\phi\in\Phi_n$ (Theorem~\ref{thm:aasd78zz}). These conditions can be interpreted in terms of symmetry of certain conditional
probabilities.

We also show (Proposition~\ref{prop:aasd78zzz}) that the condition\footnote{Note that, according to the terminology used in \cite{NavSpiBal10}, the left-hand side of (\ref{eq:wer876}) is the $k$th coordinate of the \emph{probability signature}, while the right-hand side is the $k$th coordinate of the \emph{system signature}.}
\begin{equation}\label{eq:wer876}
\Pr(T=X_{k:n})=\phi_{n-k+1}-\phi_{n-k}\, ,\qquad k\in [n]
\end{equation}
holds for every $\phi\in\Phi_n$ if and only if
\begin{equation}\label{eq:sf86}
\Pr\Big(\max_{i\in [n]\setminus A}X_i<\min_{i\in A}X_i\Big) = \frac{1}{{n\choose |A|}}~,\qquad A\subseteq [n].
\end{equation}
Finally, we show that both (\ref{eq:sdf76}) and (\ref{eq:as897ds}) hold for every $t>0$ and every $\phi\in\Phi_n$ if and only if (\ref{eq:sf86})
holds and the variables $\chi_1(t),\ldots,\chi_n(t)$ are exchangeable for every $t>0$ (Theorem~\ref{thm:aasd78zzzz}).

Through the usual identification of the elements of $\{0,1\}^n$ with the subsets of $[n]$, a pseudo-Boolean function $f\colon\{0,1\}^n\to\R$ can
be described equivalently by a set function $v_f\colon 2^{[n]}\to\R$. We simply write $v_f(A)=f(\mathbf{1}_A)$, where $\mathbf{1}_A$ denotes the
$n$-tuple whose $i$th coordinate ($i\in [n]$) is $1$, if $i\in A$, and $0$, otherwise. To avoid cumbersome notation, we henceforth use the same
symbol to denote both a given pseudo-Boolean function and its underlying set function, thus writing $f\colon\{0,1\}^n\to\R$ or $f\colon
2^{[n]}\to\R$ interchangeably.

Recall that the $k$th order statistic function $\bfx\mapsto x_{k:n}$ of $n$ Boolean variables is defined by $x_{k:n}=1$, if
$|\mathbf{x}|\geqslant n-k+1$, and $0$, otherwise. As a matter of convenience, we also formally define $x_{0:n}\equiv 0$ and $x_{n+1:n}\equiv
1$.

\section{Signature-based decomposition of the system reliability}

In the present section, without any assumption on the joint c.d.f.\ $F$, we show that, for every $t>0$, (\ref{eq:as897ds}) holds true for every
$\phi\in\Phi_n$ if and only if the state variables $\chi_1(t),\ldots,\chi_n(t)$ are exchangeable.

The following result (see Dukhovny~\cite[Thm.~2]{Duk07}) gives a useful expression for the system reliability in terms of the underlying
structure function and the component states. We provide a shorter proof here. For every $t>0$, we set
$\boldsymbol{\chi}(t)=(\chi_1(t),\ldots,\chi_n(t))$.

\begin{proposition}\label{thm:thm1}
For every $t>0$, we have
\begin{equation}\label{eq:sf676}
\overline{F}_S(t)=\sum_{\mathbf{x}\in\{0,1\}^n}\phi(\mathbf{x})\,\Pr(\boldsymbol{\chi}(t)=\mathbf{x}).
\end{equation}
\end{proposition}

\begin{proof}
We simply have
$$
\overline{F}_S(t) = \Pr(\phi(\boldsymbol{\chi}(t))=1)=\sum_{\textstyle{\bfx\in\{0,1\}^n\atop \phi(\bfx)=1}}\Pr(\boldsymbol{\chi}(t)=\mathbf{x}),
$$
which immediately leads to (\ref{eq:sf676}).
\end{proof}

Applying (\ref{eq:sf676}) to the $k$-out-of-$n$ system $\phi(\bfx)=x_{k:n}$, we obtain
$$
\overline{F}_{k:n}(t)=\sum_{|\mathbf{x}|\geqslant n-k+1}\Pr(\boldsymbol{\chi}(t)=\mathbf{x})
$$
from which we immediately derive (see \cite[Prop.~13]{DukMar08})
\begin{equation}\label{eq:sf676xy}
\overline{F}_{n-k+1:n}(t)-\overline{F}_{n-k:n}(t)=\sum_{|\mathbf{x}|=k}\Pr(\boldsymbol{\chi}(t)=\mathbf{x}).
\end{equation}

The following proposition, a key result of this paper, provides necessary and sufficient conditions on $F$ for $\overline{F}_S(t)$ to be a
certain weighted sum of the $\overline{F}_{k:n}(t)$, $k\in [n]$. We first consider a lemma.

\begin{lemma}\label{lemma:sd876}
Let $\lambda\colon\{0,1\}^n\to\R$ be a given function. We have
\begin{equation}\label{eq:d7df}
\sum_{\mathbf{x}\in\{0,1\}^n}\lambda(\bfx)\,\phi(\mathbf{x})=0\qquad \mbox{for every}~\phi\in\Phi_n
\end{equation}
if and only if $\lambda(\bfx)=0$ for all $\bfx\neq\mathbf{0}$.
\end{lemma}

\begin{proof}
Condition (\ref{eq:d7df}) defines a system of linear equations with the $2^n$ unknowns $\lambda(\bfx)$, $\bfx\in \{0,1\}^n$. We observe that
there exist $2^n-1$ functions $\phi_A\in\Phi_n$, $A\not=\varnothing$, which are linearly independent when considered as real functions (for
details, see Appendix~\ref{app:lemma}). It follows that the vectors of their values are also linearly independent. Therefore the equations in
(\ref{eq:d7df}) corresponding to the functions $\phi_A$, $A\not=\varnothing$, are linearly independent and hence the system has a rank at least
$2^n-1$. This shows that its solutions are multiples of the immediate solution $\lambda_0$ defined by $\lambda_0(\bfx)=0$, if
$\bfx\not=\mathbf{0}$, and $\lambda_0(\mathbf{0})=1$.
%
\end{proof}

Let $w\colon\{0,1\}^n\to\R$ be a given function. For every $k\in [n]$ and every $\phi\in\Phi_n$, define
\begin{equation}\label{eq:sdf65}
\phi_k^w=\sum_{|\bfx|=k}w(\bfx)\, \phi(\bfx).
\end{equation}

\begin{proposition}\label{lemma:aasd78y}
For every $t>0$, we have
$$
\overline{F}_S(t)=\sum_{k=1}^n \big(\phi^w_{n-k+1}-\phi^w_{n-k}\big)\,\overline{F}_{k:n}(t)\qquad \mbox{for every}~\phi\in\Phi_n
$$
if and only if
\begin{equation}\label{eq:sad75fdf}
\Pr(\boldsymbol{\chi}(t)=\mathbf{x}) ~=~ w(\bfx)\,\sum_{|\mathbf{z}|=|\bfx|}\Pr(\boldsymbol{\chi}(t)=\mathbf{z})\qquad \mbox{for
every}~\bfx\neq\mathbf{0}.
\end{equation}
\end{proposition}

\begin{proof}
First observe that we have
\begin{equation}\label{eq:dfrt778}
\sum_{k=1}^n
\big(\phi^w_{n-k+1}-\phi^w_{n-k}\big)\,\overline{F}_{k:n}(t)=\sum_{k=1}^n\phi^w_k\,\big(\overline{F}_{n-k+1:n}(t)-\overline{F}_{n-k:n}(t)\big).
\end{equation}
This immediately follows from the elementary algebraic identity
$$
\sum_{k=1}^na_k\, (b_{n-k+1}-b_{n-k})=\sum_{k=1}^n b_k\, (a_{n-k+1}-a_{n-k})
$$
which holds for all real tuples $(a_0,a_1,\ldots,a_n)$ and $(b_0,b_1,\ldots,b_n)$ such that $a_0=b_0=0$. Combining (\ref{eq:sf676xy}) with
(\ref{eq:sdf65}) and (\ref{eq:dfrt778}), we then obtain
\begin{eqnarray*}
\sum_{k=1}^n \big(\phi^w_{n-k+1}-\phi^w_{n-k}\big)\,\overline{F}_{k:n}(t)
&=& \sum_{k=1}^n\sum_{|\mathbf{x}|=k}w(\bfx)\,\phi(\mathbf{x})\,\sum_{|\mathbf{z}|=k}\Pr(\boldsymbol{\chi}(t)=\mathbf{z})\\
&=& \sum_{\mathbf{x}\in\{0,1\}^n} w(\bfx)\,\phi(\mathbf{x})\,\sum_{|\mathbf{z}|=|\bfx|}\Pr(\boldsymbol{\chi}(t)=\mathbf{z}).
\end{eqnarray*}
The result then follows from Proposition~\ref{thm:thm1} and Lemma~\ref{lemma:sd876}.
\end{proof}

\begin{remark}
We observe that the existence of a c.d.f.\ $F$ satisfying (\ref{eq:sad75fdf}) with $\Pr(\boldsymbol{\chi}(t)=\mathbf{x})>0$ for some $\bfx\neq\mathbf{0}$ is only possible when $\sum_{|\mathbf{z}|=|\mathbf{x}|}w(\mathbf{z})=1$. In this paper we will actually make use of (\ref{eq:sad75fdf}) only when this condition holds (see (\ref{eq:ds67dd}) and (\ref{eq:yx6czz})).
\end{remark}

We now apply Proposition~\ref{lemma:aasd78y} to obtain necessary and sufficient conditions on $F$ for (\ref{eq:as897ds}) to hold for every
$\phi\in\Phi_n$.

\begin{theorem}\label{thm:aasd78}
For every $t>0$, the representation (\ref{eq:as897ds}) holds for every $\phi\in\Phi_n$ if and only if the indicator variables
$\chi_1(t),\ldots,\chi_n(t)$ are exchangeable.
\end{theorem}

\begin{proof}
Using (\ref{eq:phikaa}) and  Proposition~\ref{lemma:aasd78y}, we see that condition (\ref{eq:as897ds}) is equivalent to
\begin{equation}\label{eq:ds67dd}
\Pr(\boldsymbol{\chi}(t)=\mathbf{x}) ~=~ \frac{1}{{n\choose |\bfx|}}\,\sum_{|\mathbf{z}|=|\bfx|}\Pr(\boldsymbol{\chi}(t)=\mathbf{z}).
\end{equation}
Equivalently, we have $\Pr(\boldsymbol{\chi}(t)=\mathbf{x})=\Pr(\boldsymbol{\chi}(t)=\mathbf{x}')$ for every $\bfx,\bfx'\in\{0,1\}^n$ such that
$|\bfx|=|\bfx'|$. This condition clearly means that $\chi_1(t),\ldots,\chi_n(t)$ are exchangeable.
\end{proof}

The following well-known proposition (see for instance \cite[Chap.~1]{Spi01} and \cite[Section~2]{Duk07}) yields an interesting interpretation of the exchangeability of the component states $\chi_1(t),\ldots,\chi_n(t)$. For the sake of self-containment, a proof is given here.

\begin{proposition}
For every $t>0$, the component states $\chi_1(t),\ldots,\chi_n(t)$ are exchangeable if and only if the probability that a group of components
survives beyond $t$ (i.e., the reliability of this group at time $t$) depends only on the number of components in the group.
\end{proposition}

\begin{proof}
Let $A\subseteq [n]$ be a group of components. The exchangeability of the component states
means that, for every $B\subseteq [n]$, the probability $\Pr(\boldsymbol{\chi}(t)=\mathbf{1}_B)$ depends only on $|B|$. In this case, the probability that the group $A$ survives beyond $t$, that is
$$
\overline{F}_A(t)=\sum_{B\supseteq A}\Pr(\boldsymbol{\chi}(t)=\mathbf{1}_B),
$$
depends only on $|A|$. Conversely, if $\overline{F}_B(t)$ depends only on $|B|$ for every $B\subseteq [n]$, then
$$
\Pr(\boldsymbol{\chi}(t)=\mathbf{1}_A)=\sum_{B\supseteq A} (-1)^{|B|-|A|}\,\overline{F}_B(t)
$$
depends only on $|A|$.
\end{proof}

\begin{remark}\label{rem:sd8f7}
Theorem~\ref{thm:aasd78} shows that the exchangeability of the component lifetimes is sufficient but not necessary for (\ref{eq:as897ds}) to
hold for every $\phi\in\Phi_n$ and every $t>0$. Indeed, the exchangeability of the component lifetimes entails the exchangeability of the
component states. This follows for instance from the identity (see \cite[Eq.~(6)]{DukMar08})
    $$
    \Pr(\boldsymbol{\chi}(t)=\mathbf{1}_A)=\sum_{B\subseteq A} (-1)^{|A|-|B|}\, F(t\mathbf{1}_{[n]\setminus B}+\infty\mathbf{1}_B).
    $$
    However, the converse statement is not true in general. As an example, consider the random vector $(X_1,X_2)$ which takes each of the values $(2,1)$, $(4,2)$, $(1,3)$ and $(3,4)$ with probability 1/4. The state variables $\chi_1(t)$ and $\chi_2(t)$ are exchangeable at any time $t$. Indeed, one can easily see that, for $|\bfx|=1$,
    $$
    \Pr(\boldsymbol{\chi}(t)=\bfx)=
    \begin{cases}
    1/4, & \mbox{if $t\in [1,4)$},\\
    0, & \mbox{otherwise}.
    \end{cases}
    $$
    However, the variables $X_1$ and $X_2$ are not exchangeable since, for instance,
    $$
    0=F(1.5, 2.5)\neq F(2.5,1.5)=1/4.
    $$
\end{remark}

\section{Alternative decomposition of the system reliability}

Assuming only that $F$ has no ties (i.e., $\Pr(X_i=X_j)=0$ for every $i\neq j$), we now provide necessary and sufficient conditions on $F$ for
formula (\ref{eq:sdf76}) to hold for every $\phi\in\Phi_n$, thus answering a question raised implicitly in \cite[p.~320]{NavSamBalBha08}.

Let $q\colon 2^{[n]}\to [0,1]$ be the \emph{relative quality function} (associated with $F$), which is defined as
$$
q(A)=\Pr\Big(\max_{i\in [n]\setminus A}X_i<\min_{i\in A}X_i\Big)
$$
with the convention that $q(\varnothing)=q([n])=1$ (see \cite[Section~2]{MarMatb}). By definition, $q(A)$ is the probability that the $|A|$ components having the longest lifetimes are exactly those in $A$. It then immediately follows that the function $q$ satisfies the following important property:
\begin{equation}\label{eq:sdf5sds}
\sum_{|\bfx|=k}q(\bfx)=1,\qquad k\in [n].
\end{equation}

Under the assumption that $F$ has no ties, the authors
\cite[Thm.~3]{MarMatb} proved that
\begin{equation}\label{eq:asd7}
\Pr(T=X_{k:n})=\phi^q_{n-k+1}-\phi^q_{n-k}\, ,
\end{equation}
where $\phi^q_k$ is defined in (\ref{eq:sdf65}).

Combining (\ref{eq:asd7}) with Proposition~\ref{lemma:aasd78y}, we immediately derive the following result.

\begin{theorem}\label{thm:aasd78zz}
Assume that $F$ has no ties. For every $t>0$, the representation (\ref{eq:sdf76}) holds for every $\phi\in\Phi_n$ if and only if
\begin{equation}\label{eq:yx6czz}
\Pr(\boldsymbol{\chi}(t)=\mathbf{x}) ~=~ q(\bfx)\,\sum_{|\mathbf{z}|=|\bfx|}\Pr(\boldsymbol{\chi}(t)=\mathbf{z}).
\end{equation}
\end{theorem}

Condition~(\ref{eq:yx6czz}) has the following interpretation. We first observe that, for every $A\subseteq [n]$,
$$
\boldsymbol{\chi}(t)=\mathbf{1}_A \quad\Leftrightarrow\quad \max_{i\in [n]\setminus A} X_i\leqslant t<\min_{i\in A}X_i\, .
$$
Assuming that $q$ is a strictly positive function, condition~(\ref{eq:yx6czz}) then means that the conditional probability
$$
 \frac{\Pr(\boldsymbol{\chi}(t)=\mathbf{1}_A)}{q(A)}=\Pr\Big(\max_{i\in [n]\setminus A} X_i\leqslant t<\min_{i\in A}X_i\,\Big|\, \max_{i\in [n]\setminus A} X_i<\min_{i\in A}X_i\Big)
$$
depends only on $|A|$.

\begin{remark}\label{remarkweak}
The concept of weak exchangeability was introduced in Navarro et al.~\cite[p.~320]{NavSamBalBha08} as follows. A random vector
$(X_1,\ldots,X_n)$ is said to be \emph{weakly exchangeable} if
\[
\Pr(X_{k:n}\leqslant t)=\Pr(X_{k:n}\leqslant t \mid X_{\sigma(1)}<\cdots<X_{\sigma(n)}),
\]
for every $t>0$, every $k\in [n]$, and every permutation $\sigma$ on $[n]$. Theorem 3.6 in \cite{NavSamBalBha08} states that if $F$ has no ties
and $(X_1,\ldots,X_n)$ is weakly exchangeable, then (\ref{eq:sdf76}) holds for every $\phi\in\Phi_n$. By Theorem~\ref{thm:aasd78zz}, we see that
weak exchangeability implies condition (\ref{eq:yx6czz}) whenever $F$ has no ties. However, the converse is not true in general. Indeed, in the
example of Remark~\ref{rem:sd8f7}, we can easily see that condition (\ref{eq:yx6czz}) holds, while the lifetimes $X_1$ and $X_2$ are not weakly
exchangeable.
\end{remark}

We now investigate condition (\ref{eq:wer876}) under the sole assumption that $F$ has no ties. Navarro and Rychlik~\cite[Lemma~1]{NavRyc07} (see
also \cite[Rem.~4]{MarMatb}) proved that this condition holds for every $\phi\in\Phi_n$ whenever the component lifetimes $X_1,\ldots,X_n$ are
exchangeable. The following proposition gives a necessary and sufficient condition on $F$ (in terms of the function $q$) for (\ref{eq:wer876})
to hold for every $\phi\in\Phi_n$.

The function $q$ is said to be \emph{symmetric} if $q(\bfx)=q(\bfx')$ whenever $|\bfx|=|\bfx'|$. By (\ref{eq:sdf5sds}) it follows that $q$ is symmetric if and only if $q(\bfx)=1/{n\choose|\bfx|}$ for every $\bfx\in\{0,1\}^n$.

\begin{proposition}\label{prop:aasd78zzz}
Assume that $F$ has no ties. Condition (\ref{eq:wer876}) holds for every $\phi\in\Phi_n$ if and only if $q$ is symmetric.
\end{proposition}

\begin{proof}
By (\ref{eq:asd7}) we have
\[
\Pr(T=X_{k:n})=\sum_{\bfx\in\{0,1\}^n}\big(\delta_{|\bfx|,n-k+1}-\delta_{|\bfx|,n-k}\big)\, q(\bfx)\,\phi(\bfx),\qquad k\in [n],
\]
where $\delta$ stands for the Kronecker delta. Similarly, by (\ref{eq:phikaa}) we have
\[
\phi_{n-k+1}-\phi_{n-k}=\sum_{\bfx\in\{0,1\}^n}\big(\delta_{|\bfx|,n-k+1}-\delta_{|\bfx|,n-k}\big)\, \frac{1}{{n\choose
|\bfx|}}\,\phi(\bfx),\qquad k\in [n].
\]
The result then follows from Lemma \ref{lemma:sd876}.
\end{proof}

We end this paper by studying the special case where both conditions (\ref{eq:sdf76}) and (\ref{eq:as897ds}) hold. We have the following result.

\begin{theorem}\label{thm:aasd78zzzz}
Assume that $F$ has no ties. The following assertions are equivalent.
\begin{enumerate}
\item[$(i)$]  Conditions (\ref{eq:sdf76}) and (\ref{eq:as897ds}) hold for every $\phi\in\Phi_n$ and every $t>0$.

\item[$(ii)$] Condition (\ref{eq:wer876}) holds for every $\phi\in\Phi_n$ and the variables $\chi_1(t),\ldots,\chi_n(t)$ are exchangeable for
every $t>0$.

\item[$(iii)$] The function $q$ is symmetric and the variables $\chi_1(t),\ldots,\chi_n(t)$ are exchangeable for every $t>0$.
\end{enumerate}
\end{theorem}

\begin{proof}
$(ii)\Leftrightarrow (iii)$ Follows from Proposition~\ref{prop:aasd78zzz}.

$(ii)\Rightarrow (i)$ Follows from Theorem~\ref{thm:aasd78}.

$(i)\Rightarrow (iii)$ By Theorem~\ref{thm:aasd78}, we only need to prove that $q$ is symmetric. Combining (\ref{eq:ds67dd}) with
(\ref{eq:yx6czz}), we obtain
$$
\bigg(q(\bfx)-\frac{1}{{n\choose |\bfx|}}\bigg)\,\sum_{|\mathbf{z}|=|\bfx|}\Pr(\boldsymbol{\chi}(t)=\mathbf{z})=0.
$$
To conclude, we only need to prove that, for every $k\in [n-1]$, there exists $t>0$ such that
$$
\sum_{|\mathbf{z}|=k}\Pr(\boldsymbol{\chi}(t)=\mathbf{z})>0.
$$
Suppose that this is not true. By (\ref{eq:sf676xy}), there exists $k\in [n-1]$ such that
$$
0=\overline{F}_{n-k+1:n}(t)-\overline{F}_{n-k:n}(t)=\Pr(X_{n-k:n}\leqslant t< X_{n-k+1:n})
$$
for every $t>0$. Then, denoting the set of positive rational numbers by $\Q^+$, the sequence of events
$$
E_m=(X_{n-k:n}\leqslant t_m< X_{n-k+1:n}),\qquad m\in\N,
$$
where $\{t_m : m\in\N\}=\Q^+$, satisfies $\Pr(E_m)=0$. Since $\Q^+$ is dense in $(0,\infty)$, we obtain
$$
\Pr(X_{n-k:n}< X_{n-k+1:n})=\Pr\bigg(\bigcup_{m\in\N}\, E_m\bigg)=0,
$$
which contradicts the assumption that $F$ has no ties.
\end{proof}

The following two examples show that neither of the conditions (\ref{eq:sdf76}) and (\ref{eq:as897ds}) implies the other.

\begin{example}
Let $(X_1,X_2,X_3)$ be the random vector which takes the values $(1, 2, 3)$, $(1, 3, 2)$, $(2, 1, 3)$, $(2, 3, 1)$, $(3, 2, 1)$, $(3, 1, 2)$, with probabilities $p_1,\ldots,p_6$, respectively. It was shown in \cite[Example~3.7]{NavSamBalBha08} that (\ref{eq:sdf76}) holds for every $\phi\in\Phi_n$ and every $t>0$. However, we can easily see that $\chi_1(t)$, $\chi_2(t)$, $\chi_3(t)$ are exchangeable for every $t>0$ if and only if $(p_1,\ldots,p_6)$ is a convex combination of $(0,1/3,1/3,0,1/3,0)$ and $(1/3,0,0,1/3,0,1/3)$. Hence, when the latter condition is not satisfied, (\ref{eq:as897ds}) does not hold for every $\phi\in\Phi_n$ by Theorem~\ref{thm:aasd78}.
\end{example}

\begin{example}
Let $(X_1,X_2,X_3)$ be the random vector which takes the values $(1, 2, 4)$, $(2, 4, 5)$, $(3, 1, 2)$, $(4, 2, 3)$, $(5, 3, 4)$, $(2, 3, 1)$, $(3, 4, 2)$, $(4, 5, 3)$ with probabilities $p_1=\cdots =p_8=1/8$. We have
$$
q(\{1\})=q(\{2\})=q(\{1,2\})=q(\{1,3\})=3/8\quad\mbox{and}\quad q(\{3\})=q(\{2,3\})=2/8,
$$
which shows that $q$ is not symmetric. However, we can easily see that $\chi_1(t)$, $\chi_2(t)$, $\chi_3(t)$ are exchangeable for every $t>0$. Indeed, we have
    $$
    \Pr(\boldsymbol{\chi}(t)=\bfx)=
    \begin{cases}
    1/8, & \mbox{if $t\in [\alpha,\beta)$},\\
    0, & \mbox{otherwise},
    \end{cases}
    $$
where $(\alpha,\beta)=(2,5)$ whenever $|\bfx|=1$ and $(\alpha,\beta)=(1,4)$ whenever $|\bfx|=2$. Thus (\ref{eq:as897ds}) holds for every $\phi\in\Phi_n$ and every $t>0$ by Theorem~\ref{thm:aasd78}. However, (\ref{eq:sdf76}) does not hold for every $\phi\in\Phi_n$ and every $t>0$ by Theorem~\ref{thm:aasd78zzzz}.
\end{example}

\begin{remark}
Let $\Phi'_n$ be the class of structure functions of $n$-component \emph{semicoherent} systems, that is, the class of nondecreasing functions
$\phi\colon\{0,1\}^n\to\{0,1\}$ satisfying the boundary conditions $\phi(\mathbf{0})=0$ and $\phi(\mathbf{1})=1$. It is clear that
Proposition~\ref{thm:thm1} and Lemma~\ref{lemma:sd876} still hold, even for $n=2$, if we extend the set $\Phi_n$ to $\Phi'_n$ (in the proof of
Lemma~\ref{lemma:sd876} it is then sufficient to consider the $2^n-1$ functions $\phi_A(\bfx)=\prod_{i\in A}x_i$, $A\neq\varnothing$). We then
observe that Propositions~\ref{lemma:aasd78y} and \ref{prop:aasd78zzz} and Theorems~\ref{thm:aasd78}, \ref{thm:aasd78zz}, and
\ref{thm:aasd78zzzz} (which use Proposition~\ref{thm:thm1} and Lemma~\ref{lemma:sd876} to provide conditions on $F$ for certain identities to
hold for every $\phi\in\Phi_n$) are still valid for $n\geqslant 2$ if we replace $\Phi_n$ with $\Phi'_n$ (that is, if we consider semicoherent
systems instead of coherent systems only). This observation actually strengthens these results. For instance, from Theorem~\ref{thm:aasd78} we
can state that, for every fixed $t>0$, if (\ref{eq:as897ds}) holds for every $\phi\in\Phi_n$, then the variables $\chi_1(t),\ldots,\chi_n(t)$
are exchangeable; conversely, for every $n\geqslant 2$ and every $t>0$, the latter condition implies that (\ref{eq:as897ds}) holds for every
$\phi\in\Phi'_n$. We also observe that the ``semicoherent'' version of Theorem~\ref{thm:aasd78} (i.e., where $\Phi_n$ is replaced with
$\Phi'_n$) was proved by Dukhovny \cite[Thm.~4]{Duk07}.
\end{remark}

\section*{Acknowledgments}

The authors wish to thank M.~Couceiro, G.~Peccati, and F.~Spizzichino for fruitful discussions. Jean-Luc Marichal and Pierre Mathonet are supported by
the internal research project F1R-MTH-PUL-09MRDO of the University of Luxembourg. Tam\'as Waldhauser is supported by the National Research Fund
of Luxembourg, the Marie Curie Actions of the European Commission (FP7-COFUND), and the Hungarian National Foundation for Scientific Research
under grant no. K77409.

\appendix
\section{}
\label{app:lemma}

In this appendix we construct $2^n-1$ functions in $\Phi_n$ which are linearly independent when considered as real functions. Here the
assumption $n\geqslant 3$ is crucial.

Assume first that $n\neq 4$ and let $\pi$ be the permutation on $[n]$ defined by the following cycles
$$
\pi =%
\begin{cases}
(1,2,\ldots,n), & \mbox{if $n$ is odd},\\
(1,2,3)\circ(4,5,\ldots,n), & \mbox{if $n$ is even}.
\end{cases}
$$
With every $A\varsubsetneq [n]$, $A\neq\varnothing$, we associate $A^*\subseteq [n]$ in the following way:
\begin{itemize}
\item if $|A|\leqslant n-2$, then we choose any set $A^*$ such that $|A^*|=n-1$ and $A\cup A^*=[n]$;

\item if $A=[n]\setminus\{k\}$ for some $k\in [n]$, then we take $A^*=[n]\setminus\{\pi(k)\}$.
\end{itemize}

We now show that the $2^n-1$ functions $\phi_A\in\Phi_n$, $A\subseteq [n]$, $A\neq\varnothing$, defined by
$$
\phi_A(\bfx)=
\begin{cases}
\big(\prod_{i\in A}x_i\big)\amalg\big(\prod_{i\in A^*}x_i\big), & \mbox{if $A\neq [n]$},\\
\prod_{i\in [n]}x_i\, , & \mbox{if $A=[n]$},
\end{cases}
$$
where $\amalg$ denotes the coproduct (i.e., $x\amalg y=x+y-xy$), are linearly independent when considered as real functions.

Suppose there exist real numbers $c_A$, $A\subseteq [n]$, $A\neq\varnothing$, such that
$$
\sum_{A\neq\varnothing} c_A\,\phi_A=0.
$$
Expanding the left-hand side of this equation as a linear combination of the functions $\prod_{i\in B}x_i$, $B\subseteq [n]$,
$B\neq\varnothing$, we first see that, if $|A|\leqslant n-2$, the coefficient of $\prod_{i\in A}x_i$ is $c_A$ and hence $c_A=0$ whenever
$0<|A|\leqslant n-2$. Next, considering the coefficient of $\prod_{i\in A}x_i$ for $A=[n]\setminus\{k\}$, $k\in [n]$, we obtain
$$
c_{[n]\setminus\{k\}}+c_{[n]\setminus\{\pi^{-1}(k)\}}=0.
$$
Since $\pi$ is made up of odd-length cycles only, it follows that $c_A=0$ whenever $|A|=n-1$.

For $n=4$ we consider the function $\pi\colon [4]\to [4]$ defined by $\pi(1)=\pi(4)=2$, $\pi(2)=3$, and $\pi(3)=4$, and choose the functions
$\phi_A$ as above. We then easily check that these functions are linearly independent.

\end{document}